\documentclass[a4paper]{article}
\usepackage{fullpage}
\usepackage{amsmath}
\usepackage{amssymb}
\usepackage{amsthm}
\usepackage{graphicx}
\usepackage{bbm}
\usepackage{enumerate}
\usepackage{microtype}
\usepackage{tocbibind}
\usepackage{xcolor}
\usepackage[none]{hyphenat}
\usepackage[colorlinks=true,urlcolor=blue,linkcolor=blue,plainpages=false,pdfpagelabels]{hyperref}
\allowdisplaybreaks

\newtheorem{theorem}{Theorem}[section]
\newtheorem{proposition}[theorem]{Proposition}

\newtheorem{remark}[theorem]{Remark}

\newtheorem{definition}[theorem]{Definition}

\newcommand{\N}{\mathbb{N}}

\newcommand{\R}{\mathbb{R}}

\newcommand{\eps}{\varepsilon}
\newcommand{\vp}{\varphi}
\newcommand{\se}{\subseteq}

\newcommand{\id}{\mathrm{id}}

\title{The computational content of super strongly nonexpansive mappings and uniformly monotone operators}
\author{Andrei Sipo\c s${}^{a,b,c}$\\[2mm]
\footnotesize ${}^a$Research Center for Logic, Optimization and Security (LOS), Department of Computer Science,\\
\footnotesize Faculty of Mathematics and Computer Science, University of Bucharest,\\
\footnotesize Academiei 14, 010014 Bucharest, Romania\\[1mm]
\footnotesize ${}^b$Simion Stoilow Institute of Mathematics of the Romanian Academy,\\
\footnotesize Calea Grivi\c tei 21, 010702 Bucharest, Romania\\[1mm]
\footnotesize ${}^c$Institute for Logic and Data Science,\\
\footnotesize Popa Tatu 18, 010805 Bucharest, Romania\\[2mm]
\footnotesize Email: andrei.sipos@fmi.unibuc.ro\\
}
\date{}

\begin{document}

\maketitle

\begin{abstract}
Recently, Liu, Moursi and Vanderwerff have introduced the class of super strongly nonexpansive mappings as a counterpart to operators which are maximally monotone and uniformly monotone. We give a  quantitative study of these notions in the style of proof mining, providing a modulus of super strong nonexpansiveness, giving concrete examples of it and connecting it to moduli associated to uniform monotonicity. For the supercoercive case, we analyze the situation further, yielding a quantitative inconsistent feasibility result for this class (obtaining effective uniform rates of asymptotic regularity), a result which is also qualitatively new.


\end{abstract}

\section{Introduction}

In 2003, Bauschke \cite{Bau03} solved in the positive the `zero displacement conjecture' of convex optimization, which states that, given a Hilbert space $X$, $m \in \N^*$ and $C_1,\ldots C_m \se X$ closed, convex, nonempty sets and setting $R:= P_{C_m} \circ \ldots \circ P_{C_1}$, we have that, for every $x \in X$,
$$\lim_{n \to \infty} \|R^nx- R^{n+1}x\|=0,$$
that is, $R$ is {\it asymptotically regular} in the sense of Browder and Petryshyn \cite[Definition 1]{BroPet66}. This result was later extended from $R$ being a composition of projections to it being a composition of firmly nonexpansive mappings \cite{BauMarMofWan12} and then of averaged mappings \cite{BauMou20}, each such mapping $R_i$ being assumed to have the {\it approximate fixed point property}, i.e. for every $\eps>0$ there is a $p \in X$ such that $\|p-R_ip\|\leq\eps$. An extensive convex optimization monograph, containing facts on all these classes of mappings, is \cite{BauCom17}.

In 2019, Kohlenbach started the study of the above results from the point of view of {\it proof mining}, which is a research paradigm (first suggested by Kreisel in the 1950s under the name of `unwinding of proofs' and then highly developed starting in the 1990s by Kohlenbach and his students and collaborators) that aims to analyze ordinary mathematical proofs using tools from proof theory (a branch of mathematical logic) in order to extract additional -- typically quantitative -- information out of them (for more information, see Kohlenbach's monograph \cite{Koh08} and his recent research survey \cite{Koh19}). In the cases above, a typical piece of information one would like to obtain is a {\it rate of asymptotic regularity}, that is, for every $\eps>0$, an effective bound on the $N$ (depending on $\eps$ and maybe some other parameters) such that for all $n \geq N$, $\|R^nx - R^{n+1}x\| \leq \eps$.

The proofs of those theorems, in order to establish asymptotic regularity, used the property of $R$ of being {\it strongly nonexpansive}, a notion first introduced by Bruck and Reich in \cite{BruRei77}, and which had already been analyzed from the point of view of proof mining by Kohlenbach in 2016 \cite{Koh16}. The way was, thus, open to find the desired rates of asymptotic regularity, and this was done by Kohlenbach for the cases of projections and firmly nonexpansive mappings in \cite{Koh19b}; later, this was extended to the averaged mappings case by the author in \cite{Sip22}.

It was a bit surprising at the time that these results could be analyzed at all from the point of view of proof mining, since the proofs used quite deep results of monotone operator theory (which relied on principles regarded as proof-theoretically strong), like the Minty theorem or the Br\'ezis-Haraux theorem \cite{BreHar76}, but analyzed they were, and the extracted rates turned out to be not too complex -- for example, in the case of projections, the rate was polynomial of degree eight. (This is an instance of a broader phenomenon called `proof-theoretic tameness', discussed in \cite{Koh20}.) Recently, this received a logical explanation through the thorough proof-theoretic treatments of monotone operator theory in \cite{Pis22} and of the Br\'ezis-Haraux theorem specifically in \cite{KohPis}. The notion underlying all these asymptotic regularity proofs was that of {\it rectangularity} of an operator (first introduced by Br\'ezis and Haraux, also in \cite{BreHar76}, and named as such by Simons \cite{Sim06}), which was given in the latter paper the quantitative form of a {\it modulus of uniform rectangularity}.

The way this notion of rectangularity had been invoked in the proof of the case of averaged mappings in \cite{BauMou20} had been highly conceptual, as the proof made use of the fact that if $A$ is a maximally monotone operator, then its {\it reflected resolvent}, $R_A$, is an averaged mapping iff $A^{-1}$ is strongly monotone, and thus, in this case, $A^{-1}$ (or, equivalently, $A$ itself) is rectangular. A natural question to ask, then, is which class of mappings corresponds in this way to the class of operators which are {\it uniformly monotone}. This was answered in 2022 by Liu, Moursi and Vanderwerff in \cite{LiuMouVan22}, who introduced the class of {\it super strongly nonexpansive mappings} to fulfill this role.

{\it The goal of this paper is to analyze this class of super strongly nonexpansive mappings from the point of view of proof mining and to extend, as much as possible, the above asymptotic regularity results, qualitatively and quantitatively.}

Towards that end, in Section~\ref{sec:ssne} we define the notion of an `SSNE-modulus' (analogously to the `SNE-modulus' defined in \cite{Koh16}) as the quantitative counterpart of super strong nonexpansiveness. We also give concrete examples of SSNE-moduli in the cases of averaged mappings, contractions for large distances and strongly nonexpansive mappings on the real line, and show how they behave under composition.

Section~\ref{sec:um} is dedicated to the connection to uniformly monotone operators. We first give a quite abstract (as it also applies to other similar notions, including super strong nonexpansiveness itself) characterization of uniform monotonicity, showing how different kinds of moduli are related, in the form of Proposition~\ref{prop-m}. Then, using the most flexible kind of modulus, we show how one can go back and forth, quantitatively, between (inverse) uniform monotonicity and super strong nonexpansiveness. In the final part of the section, we obtain a modulus of uniform continuity for operators which are maximally monotone and inverse uniformly monotone, as qualitative uniform continuity had been shown in \cite[Theorem 4.5]{LiuMouVan22}.

Finally, in Section~\ref{sec:rect} we treat the case where the modulus of uniform monotonicity is supercoercive, since in that case it is known that the operator is rectangular. We refine the abstract characterization from before in order to treat this case, and we obtain a characterization of supercoercivity which is not expressed in terms of the modulus of uniform monotonicity. We then obtain the analogous subclass of the one of super strongly nonexpansive mappings, whose elements we dub {\it supercoercively super strongly nonexpansive mappings}, and show that averaged mappings and contractions for large distances belong to it. We are then able to obtain quantitative forms of rectangularity (Proposition~\ref{prop-theta}), the approximate fixed point property (Theorem~\ref{thm-psi}) and asymptotic regularity (Theorem~\ref{thm-main}) for this class of mappings. The last two results are also qualitatively new, so we summarize them without the corresponding moduli and rates in the form of Theorem~\ref{thm-qual}.

\section{Moduli for super strongly nonexpansive mappings}\label{sec:ssne}

\begin{definition}
Let $X$ be a Banach space, $S \se X$ and $T: S \to X$. We say that $T$ is {\bf nonexpansive} if, for all $x$, $y \in S$, $\|Tx-Ty\| \leq \|x-y\|$.
\end{definition}

The following definition was introduced by Bruck and Reich \cite{BruRei77}.

\begin{definition}
Let $X$ be a Banach space, $S \se X$ and $T: S \to X$. We say that $T$ is {\bf strongly nonexpansive} if it is nonexpansive and, for all $(x_n)$, $(y_n) \se S$ such that the sequence $(x_n-y_n)$ is bounded and $\lim_{n \to \infty} (\|x_n - y_n\| - \|Tx_n - Ty_n\|) = 0$, we have that $\lim_{n \to \infty} ((x_n - y_n) - (Tx_n - Ty_n)) = 0$.
\end{definition}

The following quantitative version of strong nonexpansiveness was given in \cite{Koh16}.

\begin{proposition}\label{sne-eq}
Let $X$ be a Banach space, $S \se X$ and $T: S \to X$. TFAE:
\begin{enumerate}[(a)]
\item $T$ is strongly nonexpansive;
\item there exists a function $\omega:(0,\infty) \times (0,\infty) \to (0,\infty)$ (called an {\bf SNE-modulus}) such that, for any $b$, $\eps > 0$ and $x$, $y \in S$ with $\|x-y\| \leq b$ and $\|x-y\| - \|Tx-Ty\| < \omega(b,\eps)$, we have that $\|(x-y) - (Tx-Ty)\| < \eps$.
\end{enumerate}
\end{proposition}

\begin{proof}
See \cite[Lemma 2.2]{Koh16} and \cite[Remark 2]{Koh19b}.
\end{proof}

The following definition was recently introduced by Liu, Moursi and Vanderwerff \cite{LiuMouVan22}.

\begin{definition}[{\cite[Definition 3.1]{LiuMouVan22}}]
Let $X$ be a Banach space, $S \se X$ and $T: S \to X$. We say that $T$ is {\bf super strongly nonexpansive} if it is nonexpansive and, for all $(x_n)$, $(y_n) \se S$ such that $\lim_{n \to \infty} (\|x_n - y_n\|^2 - \|Tx_n - Ty_n\|^2) = 0$, we have that $\lim_{n \to \infty} ((x_n - y_n) - (Tx_n - Ty_n)) = 0$.
\end{definition}

\begin{remark}
Note that there is no boundedness condition in the above definition.
\end{remark}

We now give the analogue of Proposition~\ref{sne-eq} for super strongly nonexpansive mappings.

\begin{proposition}\label{ssne-eq}
Let $X$ be a Banach space, $S \se X$ and $T: S \to X$. TFAE:
\begin{enumerate}[(a)]
\item $T$ is super strongly nonexpansive;
\item there exists a function $\chi:(0,\infty) \to (0,\infty)$ (called an {\bf SSNE-modulus}) such that, for any $\eps > 0$ and $x$, $y \in S$ with $\|x-y\|^2 - \|Tx-Ty\|^2 < \chi(\eps)$, we have that $\|(x-y) - (Tx-Ty)\| < \eps$.
\end{enumerate}
\end{proposition}

\begin{proof}
The proof largely follows the lines of \cite[Lemma 2.2]{Koh16} and \cite[Remark 2]{Koh19b}, but we give it below for completeness.

For `$(a)\Rightarrow(b)$', we may assume that there is an $\eps>0$ such that, for all $\delta > 0$, there are $x$, $y \in S$ with $\|x-y\|^2 - \|Tx-Ty\|^2 < \delta$ and $\|(x-y) - (Tx-Ty)\| \geq \eps$. For if this were not so, we may define a function $\chi:(0,\infty) \to (0,\infty)$ by choosing for every $\eps>0$ a $\delta>0$ such that for all $x$, $y \in S$ with $\|x-y\|^2 - \|Tx-Ty\|^2 < \delta$, we have that $\|(x-y) - (Tx-Ty)\| < \eps$, which is the kind of function that we have to construct.

By our assumption, we may construct sequences $(x_n)$, $(y_n) \se S$ by putting, for any $n \in \N$, $x_n$ and $y_n \in S$ to be such that $\|x_n-y_n\|^2 - \|Tx_n-Ty_n\|^2 < \frac1{n+1}$ and $\|(x_n-y_n) - (Tx_n-Ty_n)\| \geq \eps$. But, then, $\lim_{n \to \infty} (\|x_n - y_n\|^2 - \|Tx_n - Ty_n\|^2) = 0$, while $((x_n - y_n) - (Tx_n - Ty_n))$ does not converge to $0$, a contradiction.

We now show `$(b)\Rightarrow(a)$'. We first have to show that $T$ is nonexpansive. Assume that there are $x$, $y \in S$ such that $\|Tx-Ty\| > \|x-y\|$, so $\|Tx-Ty\|^2 > \|x-y\|^2$ and $\|x-y\|^2 - \|Tx-Ty\|^2 <0$. Then, for every $\eps >0$, we have that $\|x-y\|^2 - \|Tx-Ty\|^2 < \chi(\eps)$, so $\|(x-y) - (Tx-Ty)\| < \eps$. Thus, $x-y=Tx-Ty$, so $\|Tx-Ty\|\leq\|x-y\|$.

For the second condition, let $(x_n)$, $(y_n) \se S$ be such that $\lim_{n \to \infty} (\|x_n - y_n\|^2 - \|Tx_n - Ty_n\|^2) = 0$. We have that for all $\eps>0$ there is an $N \in \N$ such that for all $n \geq N$, $\|x_n - y_n\|^2 - \|Tx_n - Ty_n\|^2 < \chi(\eps)$, so $\|(x_n - y_n) - (Tx_n - Ty_n)\| < \eps$. Thus, we have that $\lim_{n \to \infty} ((x_n - y_n) - (Tx_n - Ty_n)) = 0$.
\end{proof}

The following is a quantitative version of \cite[Proposition 3.2]{LiuMouVan22}, which shows that every super strongly nonexpansive mapping is strongly nonexpansive.

\begin{proposition}\label{omega-bullet}
For any $\chi:(0,\infty) \to (0,\infty)$, we define $\omega_\chi: (0,\infty) \times (0,\infty) \to (0,\infty)$, by putting, for any $b>0$ and $\eps > 0$,
$$\omega_\chi(b,\eps):=\frac{\chi(\eps)}{2b}.$$
Let $\chi:(0,\infty) \to (0,\infty)$. Let $X$ be a Banach space, $S \se X$ and $T: S \to X$ be such that $T$ is super strongly nonexpansive with modulus $\chi$. Then $T$ is strongly nonexpansive with modulus $\omega_\chi$.
\end{proposition}

\begin{proof}
Let $b$, $\eps > 0$ and $x$, $y \in S$ with $\|x-y\| \leq b$ and $\|x-y\| - \|Tx-Ty\| < \omega_\chi(b,\eps)$. We want to show that $\|(x-y) - (Tx-Ty)\| < \eps$. Since $T$ is nonexpansive, we have that
$$\|x-y\|^2 - \|Tx-Ty\|^2 = (\|x-y\|-\|Tx-Ty\|)(\|x-y\|+\|Tx-Ty\|) < \omega_\chi(b,\eps) \cdot 2b = \chi(\eps).$$
Since $\chi$ is an SSNE-modulus, it follows that $\|(x-y) - (Tx-Ty)\| < \eps$.
\end{proof}

We now proceed to give examples of super strongly nonexpansive mappings. The following is a quantitative version of \cite[Proposition 3.3]{LiuMouVan22}, which shows that, on the real line, the notions of strong nonexpansiveness and super strong nonexpansiveness coincide.

\begin{proposition}
Let $\omega:(0,\infty) \times (0,\infty) \to (0,\infty)$ and $T: \R \to \R$ be such that $T$ is strongly nonexpansive with modulus $\omega$. Then $T$ is super strongly nonexpansive with modulus
$$\chi(\eps):=\min\left((\omega(1,\eps))^2,\eps^2,(\omega(1,1/2))^2,1\right).$$
\end{proposition}

\begin{proof}
Assume, towards a contradiction, that there are $\eps>0$ and $x$, $y \in \R$ such that $|x-y|^2 - |Tx-Ty|^2 < \chi(\eps)$ and $|(x-y) - (Tx-Ty)| \geq \eps$. It is immediate that, for any $a$, $b \in \R$, $(|a| - |b|)^2 \leq ||a|^2 - |b|^2|$. Thus, using also that $T$ is nonexpansive, we get that
$$(|x-y| - |Tx-Ty|)^2 \leq |x-y|^2 - |Tx-Ty|^2 < \chi(\eps),$$
and that $0 \leq |x-y| - |Tx-Ty| < \sqrt{\chi(\eps)} \leq \omega(1,\eps)$. Therefore, since $|(x-y) - (Tx-Ty)| \geq \eps$ and $\omega$ is an SNE-modulus for $T$, we must have $|x-y| > 1$. We may assume w.l.o.g. that $y-x>1$.

If we had $Ty-Tx \geq 0$, then we would have $0 \leq (y-x) - (Ty-Tx) < \sqrt{\chi(\eps)} $, so $|(x-y) - (Tx-Ty)| < \sqrt{\chi(\eps)} \leq \eps$, contradicting our assumption. Thus, $Ty-Tx < 0$ and $(y-x) - (Tx-Ty) < \sqrt{\chi(\eps)}$.

Set $z:=x+1$, so $x<z<y$ and $|x-z| + |z-y| = |x-y|$. Then, since $T$ is nonexpansive,
\begin{align*}
0 &\leq |x-z| - |Tx-Tz| \leq |x-z| - |Tx-Tz| + |z-y| - |Tz-Ty| = |x-y| - |Tx-Tz| - |Tz-Ty|\\
&\leq |x-y| - |Tx-Ty| < \sqrt{\chi(\eps)} \leq \omega(1,1/2).
\end{align*}
Therefore, we must have that $|(x-z) - (Tx-Tz)| < 1/2$, i.e. that $|1-(Tz-Tx) < 1/2$, so $Tz \geq Tx$.

On the other hand, we have that $Tz-Ty \leq |Tz-Ty| \leq |z-y| = y-x-1$, so $Tz-(y-x-1) \leq Ty$. Since $Tz \geq Tx$, we have that $Tx - (y-x-1) \leq Ty$, i.e. that $(y-x) - (Tx-Ty) \geq 1$. But we have shown that $(y-x) - (Tx-Ty) < \sqrt{\chi(\eps)} \leq 1$, which gives us the required contradiction.
\end{proof}

\begin{definition}
Let $X$ be a Banach space, $\alpha \in (0,1)$ and $R:X\to X$. We say that $R$ is {\bf $\alpha$-averaged} if there is a nonexpansive mapping $T:X \to X$ such that $R=(1-\alpha)\id_X + \alpha T$.
\end{definition}

In order to show that, in Hilbert spaces, every averaged mapping is super strongly nonexpansive (a fact not explicitly stated in \cite{LiuMouVan22}), we shall need the following identity.

\begin{proposition}[{\cite[Corollary 2.15]{BauCom17}}]\label{h-id}
Let $X$ be a Hilbert space, $a$, $b \in X$ and $\lambda \in (0,1)$. Then
$$\|(1-\lambda)a + \lambda b\|^2 = (1-\lambda)\|a\|^2 + \lambda\|b\|^2 - \lambda(1-\lambda)\|a-b\|^2.$$
\end{proposition}

\begin{proposition}\label{av-ssne}
Let $X$ be a Hilbert space, $\alpha \in (0,1)$ and $R:X \to X$ be $\alpha$-averaged. Then $R$ is super strongly nonexpansive with modulus $\chi(\eps):=\eps^2 \cdot \frac{1-\alpha}{\alpha}$.
\end{proposition}

\begin{proof}
Let $T:X \to X$ be the nonexpansive mapping such that $R=(1-\alpha)\id_X + \alpha T$. Let $\eps>0$ and $x$, $y \in X$ be such that
$$\|x-y\|^2 - \|Rx-Ry\|^2 < \eps^2 \cdot \frac{1-\alpha}{\alpha}.$$
We have that $\|Rx-Ry\| = \|(1-\alpha)(x-y) + \alpha(Tx-Ty)\|$, so, by Proposition~\ref{h-id},
$$\|Rx-Ry\|^2 = (1-\alpha)\|x-y\|^2 + \alpha\|Tx-Ty\|^2 - \alpha(1-\alpha)\|(x-y)-(Tx-Ty)\|^2.$$
We get that
$$\|x-y\|^2 - \eps^2 \cdot \frac{1-\alpha}{\alpha} < (1-\alpha)\|x-y\|^2 + \alpha\|Tx-Ty\|^2 - \alpha(1-\alpha)\|(x-y)-(Tx-Ty)\|^2,$$
so
$$\alpha(1-\alpha)\|(x-y)-(Tx-Ty)\|^2 < \eps^2 \cdot \frac{1-\alpha}{\alpha} + \alpha(\|Tx-Ty\|^2 - \|x-y\|^2) \leq \eps^2 \cdot \frac{1-\alpha}{\alpha}$$
and thus $\|(x-y)-(Tx-Ty)\| < \eps/\alpha$, from which we get that $\|(x-y)-(Rx-Ry)\| =\alpha \|(x-y)-(Tx-Ty)\| < \eps$.
\end{proof}

\begin{definition}[{\cite[Definition 4.2]{LiuMouVan22}}]
Let $X$ be a Banach space, $T: X \to X$ and $K:(0,\infty) \to [0,1)$. We say that $T$ is a {\bf contraction for large distances} with modulus $K$ if for all $\eps>0$ and $x$, $y\in X$ with $\|x-y\| \geq \eps$, $\|Tx-Ty\| \leq K(\eps)\|x-y\|$.
\end{definition}

\begin{remark}
Let $X$ be a Banach space, $T: X \to X$ and $K:(0,\infty) \to [0,1)$, such that $T$ is a contraction for large distances with modulus $K$. Then $T$ is nonexpansive and $-T$ is also a contraction for large distances with modulus $K$.
\end{remark}

The following two propositions provide quantitative versions of implications of \cite[Proposition 5.3]{LiuMouVan22}.

\begin{proposition}\label{cld-ssne}
Let $X$ be a Banach space, $T: X \to X$ and $K:(0,\infty) \to [0,1)$, such that $T$ is a contraction for large distances with modulus $K$. Then $T$ is super strongly nonexpansive with modulus $\chi(\eps):= \left(1-\left(K\left(\frac\eps2\right)\right)^2\right)\left(\frac\eps2\right)^2$.
\end{proposition}

\begin{proof}
Let $\eps>0$ and $x$, $y \in X$ be such that $\|x-y\|^2 - \|Tx-Ty\|^2 < \chi(\eps)$. Set $\delta:=\eps/2$ and assume that $\|x-y\| \geq \delta$. Then $\|Tx-Ty\| \leq K(\delta)\|x-y\|$, so
$$\|x-y\|^2 - \|Tx-Ty\|^2 \geq (1-(K(\delta))^2)\|x-y\|^2 \geq (1-(K(\delta))^2)\delta^2 = \chi(\eps),$$
a contradiction. Thus, $\|x-y\|<\delta$, so $\|Tx-Ty\|<\delta$ and $\|(x-y)-(Tx-Ty)\| < 2\delta = \eps$.
\end{proof}

\begin{proposition}
Let $X$ be a Banach space and $T:X\to X$ such that $T$ is strongly nonexpansive with modulus $\omega_+$ and $-T$ is strongly nonexpansive with modulus $\omega_-$.

Then, setting $\psi(b,\eps):=\min(\omega_+(b,\eps),\omega_-(b,\eps))$, we have that $T$ is a contraction for large distances with modulus $K(\eps):=\max\left(0,1-\frac{\psi(2\eps,\eps)}{2\eps} \right) \in [0,1)$.
\end{proposition}

\begin{proof}
Assume that there are $\eps>0$ and $x$, $y \in X$ such that $\eps \leq \|x-y\|$ and $\|Tx-Ty\| > K(\eps)\|x-y\|$. Let $k \in \N$ be minimal such that there are $\eps>0$ and $x$, $y \in X$ with this property and $2^k\eps \leq \|x-y\| < 2^{k+1}\eps$. Assume that $k>0$. Let $z:=(x+y)/2$, so $x-z=z-y=(x-y)/2$ and $\|x-z\|$, $\|z-y\| \in [2^{k-1}\eps,2^k\eps)$ (taking note that $k-1\geq 0$). Therefore,
$$K(\eps)\|x-z\| + K(\eps)\|z-y\| =K(\eps)\|x-y\| < \|Tx-Ty\| \leq \|Tx-Tz\| + \|Tz-Ty\|,$$
so either $K(\eps)\|x-z\| < \|Tx-Tz\|$ or $K(\eps)\|z-y\| < \|Tz-Ty\|$, contradicting the minimality of $k$. Thus, $k=0$ and $\|x-y\|\leq 2\eps$. We have that
$$\|x-y\| - \|Tx-Ty\| < (1-K(\eps))\|x-y\| \leq (1-K(\eps)) \cdot 2\eps \leq \psi(2\eps,\eps),$$
using at the last step that $K(\eps) \geq 1-\frac{\psi(2\eps,\eps)}{2\eps}$. Now, since $\|x-y\| \leq 2\eps$ and $\|x-y\|- \|Tx-Ty\| < \omega_+(2\eps,\eps)$, we have that $\|(x-y)-(Tx-Ty)\| < \eps$. Similarly, using $-T$ and $\omega_-$, we have that $\|(x-y)+(Tx-Ty)\| <\eps$. We get that $\|2(x-y)\| \leq \|(x-y)-(Tx-Ty)\| + \|(x-y)+(Tx-Ty)\| < 2\eps$, so $\|x-y\| < \eps$, a contradiction.
\end{proof}

In \cite[Proposition 7.2.(i)]{LiuMouVan22}, it is shown that super strongly nonexpansive mappings are closed under compositions. We give below the corresponding modulus for the composition, using some ideas of \cite[Theorem 2.10]{Koh16}.

\begin{proposition}\label{prop-comp}
For any $n \geq 1$ and $\chi_1,\ldots,\chi_n : (0,\infty) \to (0,\infty)$, we define $\chi_{\chi_1,\ldots,\chi_n} : (0,\infty) \to (0,\infty)$, by putting, for any $\eps >0 $,
$$\chi_{\chi_1,\ldots,\chi_n}(\eps) := \min\{\chi_i(\eps/n) \mid 1 \leq i \leq n\}.$$
Let $X$ be a Banach space, $n \geq 1$, $\chi_1,\ldots,\chi_n : (0,\infty) \to (0,\infty)$ and $T_1,\ldots,T_n : X \to X$ be super strongly nonexpansive with moduli $\chi_1,\ldots,\chi_n$, respectively. Then $T:=T_n\circ\ldots \circ T_1$ is super strongly nonexpansive with modulus $\chi_{\chi_1,\ldots,\chi_n}$.
\end{proposition}

\begin{proof}
Let $\eps>0$ and $x$, $y \in X$ be such that $\|x-y\|^2 - \|Tx-Ty\|^2 < \chi_{\chi_1,\ldots,\chi_n}(\eps)$, i.e.
$$\sum_{i=1}^n \left(\left\|(T_{i-1} \circ \ldots\circ T_1)x - (T_{i-1} \circ \ldots \circ T_1)y\right\|^2 -  \left\|(T_i \circ \ldots\circ T_1)x - (T_i \circ \ldots \circ T_1)y\right\|^2\right) <\chi_{\chi_1,\ldots,\chi_n}(\eps).$$
For each $i \in \{1,\ldots,n\}$, since $\chi_{\chi_1,\ldots,\chi_n}(\eps) \leq \chi_i(\eps/n)$, we get that
$$0 \leq \left\|(T_{i-1} \circ \ldots\circ T_1)x - (T_{i-1} \circ \ldots \circ T_1)y\right\|^2 -  \left\|(T_i \circ \ldots\circ T_1)x - (T_i \circ \ldots \circ T_1)y\right\|^2 < \chi_i(\eps/n),$$
so
$$\|((T_{i-1} \circ \ldots\circ T_1)x - (T_{i-1} \circ \ldots \circ T_1)y) - ((T_i \circ \ldots\circ T_1)x - (T_i \circ \ldots \circ T_1)y)\| < \eps/n.$$
Thus, we derive that $\|(x-y)-(Tx-Ty)\| < \eps$.
\end{proof}

\section{Uniform monotonicity}\label{sec:um}

From now on, we shall fix a Hilbert space $X$. The following definitions are standard in convex optimization.

\begin{definition}[{see, e.g., \cite[Definition 4.1.(i)]{BauCom17}}]
A mapping $U:X\to X$ is called {\bf firmly nonexpansive} if, for all $x$, $y \in X$, $\|Ux-Uy\|^2 \leq \langle x-y,Ux-Uy\rangle$.
\end{definition}

It is known (see \cite[Proposition 4.4]{BauCom17}) that a mapping $U: X\to X$ is $(1/2)$-averaged iff it is firmly nonexpansive. We thus get a bijective correspondence between firmly nonexpansive and plainly nonexpansive operators given by $U \mapsto 2U - \id_X$.

\begin{definition}[{see, e.g., \cite[Definitions 20.1 and 20.20]{BauCom17}}]
A set-valued operator $A \subseteq X \times X$ is called {\bf monotone} if, for any $(a,b)$, $(c,d) \in A$, $\langle a-c,b-d \rangle \geq 0$; it is {\bf maximally monotone} (or {\bf maximal monotone}) if it is maximal among monotone operators as ordered by inclusion.
\end{definition}

It is obvious that if an operator $A$ is (maximally) monotone, then $A^{-1}$ is also (maximally) monotone.

\begin{definition}[{see, e.g., \cite[Section 23.2]{BauCom17}}]
If $A \subseteq X \times X$ is maximally monotone, then $(\id_X + A)^{-1}$ is a firmly nonexpansive single-valued mapping on $X$ which is denoted by $J_A$ and called the {\bf resolvent} of $A$. We define the {\bf reflected resolvent} of $A$ as $R_A:=2J_A - \id_X$.
\end{definition}

The association $A \mapsto J_A$ is bijective, by \cite[Propositions 23.8 and 23.10]{BauCom17}, and if we compose it with the previous bijection, we get that the association $A \mapsto R_A$ is bijective. In addition, we have that
\begin{equation}\label{ara}
A = \left\{ \left(\frac{x+R_Ax}2,\frac{x-R_Ax}2\right) \middle|\ x \in X \right\}.
\end{equation}
We also have that, for all $x \in X$, $x=J_Ax+J_{A^{-1}}x$, so, for all $x$, $y \in X$,
\begin{equation}\label{p1}
\langle x-y,J_Ax-J_Ay \rangle = \langle J_{A^{-1}}x - J_{A^{-1}}y, J_Ax - J_Ay \rangle + \|J_Ax - J_Ay\|^2.
\end{equation}

\begin{definition}[{see, e.g., \cite[Definition 22.1.(iii)]{BauCom17}}]\label{d-um}
We say that a set-valued operator $A \subseteq X \times X$ is {\bf uniformly monotone} if there is a function $\vp: [0,\infty) \to [0,\infty]$ which is nondecreasing and vanishes only at $0$ such that for all $(a,b)$, $(c,d) \in A$, $\langle a-c,b-d\rangle \geq \vp(\|a-c\|)$.
\end{definition}

The following technical result shall help us in our quantitative treatment of the class of uniformly monotone operators.

\begin{proposition}\label{prop-m}
For any set $M$, and any $A: M \to [0,\infty)$ and $B:M \to \R$, define the function $\vp_{(M,A,B)} : [0,\infty) \to [-\infty,\infty]$ by setting, for any $\eps \geq 0$,
$$ \vp_{(M,A,B)}(\eps) := \begin{cases} \inf\{B(y) \mid y \in M, A(y) \geq \eps \}, &\mbox{if  $\eps>0$,} \\ 
\min\left(0,\inf\{B(y) \mid y \in M\}\right), & \mbox{if $\eps=0$}. \end{cases}$$
Let $M$ be a set, $A: M \to [0,\infty)$ and $B:M \to \R$. Then:
\begin{enumerate}[(i)]
\item \begin{enumerate}[(a)]
\item $\vp_{(M,A,B)}$ is nondecreasing;
\item for all $x \in M$, $\vp_{(M,A,B)}(A(x)) \leq B(x)$;
\end{enumerate}
\item TFAE:
\begin{enumerate}[(a)]
\item $\vp_{(M,A,B)}(0)=0$ and, for all $\eps>0$, $\vp_{(M,A,B)}(\eps)>0$;
\item there is a function $\vp: [0,\infty) \to [0,\infty]$ which is nondecreasing and vanishes only at $0$ such that, for all $x \in M$, $\vp(A(x)) \leq B(x)$;
\item there is a function $\psi : (0,\infty) \to (0,\infty)$ such that, for all $x \in M$ and all $\eps>0$ with $A(x) \geq \eps$, we have that $\psi(\eps) \leq B(x)$.
\end{enumerate}
In this case, we shall say that the triple $(M,A,B)$ is {\bf adequate}, we shall call a function $\vp$ as above a {\bf classical modulus} for it and a function $\psi$ as above simply a {\bf modulus} for it.
\end{enumerate}
\end{proposition}

\begin{proof}
\begin{enumerate}[(i)]
\item Obvious.
\item For `$(a)\Rightarrow(b)$', just take $\vp:=\vp_{(M,A,B)}$.

For `$(b)\Rightarrow(c)$', define $\psi : (0,\infty) \to (0,\infty)$ by setting, for any $\eps>0$,
$$ \psi(\eps) := \begin{cases} \vp(\eps), &\mbox{if  there is an $x\in M$ with $A(x)\geq\eps$,} \\ 
1, & \mbox{otherwise}. \end{cases}$$

Now, for any $x \in M$ and $\eps >0$ with $A(x) \geq \eps$, we have that  $\psi(\eps) = \vp(\eps) \leq \vp(A(x)) \leq B(x)$. In particular, this argument also shows that $\psi$ is well-defined, i.e. that for any $\eps>0$ such that there is an $x \in M$ with $A(x) \geq \eps$, $\vp(\eps)\neq\infty$.

We now show `$(c)\Rightarrow(a)$'. The fact that $\vp_{(M,A,B)}(0)=0$ follows from the fact that for all $x \in M$, $\psi(A(x)) \leq B(x)$, so $B(x) \geq 0$. Now, let $\eps>0$ and assume, towards a contradiction, that $\vp_{(M,A,B)}(\eps)=0$ (taking into account that $\vp_{(M,A,B)}$ is nondecreasing). Then there is an $x \in M$ with $A(x) \geq \eps$ and $B(x) \leq \psi(\eps)/2$, so $\psi(\eps) \leq B(x) \leq \psi(\eps)/2$, a contradiction.
\end{enumerate}
\end{proof}

Let $A \se X \times X$. Recalling Definition~\ref{d-um}, $A$ is uniformly monotone if there is a function $\vp: [0,\infty) \to [0,\infty]$ which is nondecreasing and vanishes only at $0$ such that for all $(a,b)$, $(c,d) \in A$, $\langle a-c,b-d\rangle \geq \vp(\|a-c\|)$. Now, using the above proposition, we see that $A$ is uniformly monotone iff there is a $\psi : (0,\infty) \to (0,\infty)$ such that for all $\eps>0$ and all $(a,b)$, $(c,d) \in A$ with $\|a-c\| \geq \eps$, we have that $\psi(\eps) \leq\langle a-c,b-d\rangle$. We shall also use in this particular case the terminology of `classical modulus' and `modulus', respectively.

\begin{definition}
We say that a set-valued operator $A \subseteq X \times X$ is {\bf inverse uniformly monotone} with modulus $\psi$ (resp. with classical modulus $\vp$) if $A^{-1}$ is uniformly monotone with the same modulus $\psi$ (resp. with the same classical modulus $\vp$). 
\end{definition}

An inverse uniformly monotone operator $A$ is necessarily a single-valued mapping on the whole of $X$, since on one hand single-valuedness is trivially implied by the definition, whereas on the other hand, since $A^{-1}$ is uniformly monotone, it is -- by \cite[Theorem 4.5]{LiuMouVan22} -- surjective, yielding that $A$ has full domain. Therefore, such a single-valued mapping $A:X \to X$ is inverse uniformly monotone with modulus $\psi$ if and only if, for all $\eps>0$ and $x$, $y \in X$ with $\|Ax-Ay\| \leq \eps$, $\langle x-y,Ax-Ay \rangle \geq \psi(\eps)$.

Let $A \se X \times X$ be maximally monotone. By \cite[Proposition 3.5]{LiuMouVan22}, $A$ is inverse uniformly monotone iff $R_A$ is super strongly nonexpansive. The following propositions will make this correspondence quantitatively explicit.

\begin{proposition}\label{ium-ssne}
Let $A$ be a maximally monotone operator on $X$ which is inverse uniformly monotone with modulus $\psi$. Then $R_A$ is super strongly nonexpansive with modulus $\chi(\eps):=4\psi(\eps/2)$.
\end{proposition}

\begin{proof}
Let $\eps>0$ and $x$, $y \in X$ with $\|x-y\|^2 - \|R_Ax- R_Ay\|^2 < \chi(\eps)$. Assume that $\|(x-y)-(R_Ax-R_Ay)\| \geq \eps$, so
$$\left\|\frac{x-R_Ax}2 - \frac{y-R_Ay}2 \right\| \geq \frac\eps2.$$
Using \eqref{ara}, we get that
$$\left\langle \frac{x+R_Ax}2 - \frac{y+R_Ay}2, \frac{x-R_Ax}2 - \frac{y-R_Ay}2 \right\rangle\geq \psi(\eps/2),$$
so
$$\|x-y\|^2 - \|R_Ax- R_Ay\|^2 = \langle (x-y) + (R_Ax-R_Ay),(x-y) - (R_Ax-R_Ay) \rangle\geq 4\psi(\eps/2) = \chi(\eps),$$
a contradiction.
\end{proof}

\begin{proposition}\label{ssne-ium}
For any $\chi : (0,\infty) \to (0,\infty)$ and any $\eps>0$, put
$$\psi_\chi(\eps):=\frac{\chi(2\eps)}4.$$
Let $\chi : (0,\infty) \to (0,\infty)$ and $A$ be a maximally monotone operator on $X$ such that $R_A$ is super strongly nonexpansive with modulus $\chi$. Then $A$ is inverse uniformly monotone with modulus $\psi_\chi$.
\end{proposition}

\begin{proof}
Let $\eps>0$ and $(a,b)$, $(c,d) \in A$ with $\|b-d\| \geq \eps$. Assume that $\langle a-c,b-d \rangle < \psi(\eps)$. Using \eqref{ara}, we get that there are $x$, $y \in X$ such that
$$a=\frac{x+R_Ax}2,\ b=\frac{x-R_Ax}2,\ c=\frac{y+R_Ay}2,\ d=\frac{y-R_Ay}2,$$
so
$$\left\langle \frac{x+R_Ax}2 - \frac{y+R_Ay}2, \frac{x-R_Ax}2 - \frac{y-R_Ay}2 \right\rangle< \psi(\eps)$$
and
$$\|x-y\|^2 - \|R_Ax- R_Ay\|^2 = \langle (x-y) + (R_Ax-R_Ay),(x-y) - (R_Ax-R_Ay)\rangle < 4\psi(\eps) = \chi(2\eps).$$
We thus get that $ \|(x-y)-(R_Ax-R_Ay)\| < 2\eps$, but, since $(x-y)-(R_Ax-R_Ay) = 2(b-d)$, we have that $\|b-d\|<\eps$, a contradiction.
\end{proof}

We notice that the associations of moduli in the above two propositions are mutually inverse.

By \cite[Theorem 4.5]{LiuMouVan22}, we also have that every operator $A: X \to X$ which is maximally monotone and inverse uniformly monotone is uniformly continuous. In the sequel, we shall extract from the proof a modulus of uniform continuity. The following proposition is a quantitative version of \cite[Lemma 4.4.(i)]{LiuMouVan22}.

\begin{proposition}\label{p2}
Define, for any suitable $\psi$, $\eps$,
$$\alpha_\psi(\eps):=\min\left(\psi(\eps/2),\eps^2/4\right).$$
Let $\psi:(0,\infty) \to (0,\infty)$ and $A$ be a maximally monotone operator on $X$ which is inverse uniformly monotone with modulus $\psi$. Then $J_A$ is uniformly monotone with modulus $\alpha_\psi$.
\end{proposition}

\begin{proof}
Let $\eps>0$ and $x$, $y \in X$ with $\|x-y\| \geq \eps$. By \eqref{p1}, we have that
$$\langle x-y,J_Ax-J_Ay \rangle = \langle J_{A^{-1}}x - J_{A^{-1}}y, J_Ax - J_Ay \rangle + \|J_Ax - J_Ay\|^2.$$

If we had $\|J_Ax-J_Ay\| < \eps/2$ and $\|J_{A^{-1}}x-J_{A^{-1}}y\| < \eps/2$, then, since $x = J_Ax + J_{A^{-1}}x$ and $y = J_Ay + J_{A^{-1}}y$, we would have $\|x-y\|<\eps$, a contradiction. Thus, either $\|J_Ax-J_Ay\| \geq \eps/2$ or $\|J_{A^{-1}}x-J_{A^{-1}}y\| \geq \eps/2$.

If $\|J_Ax-J_Ay\| \geq \eps/2$, then, since $(J_Ax,J_{A^{-1}}x)$, $(J_Ay,J_{A^{-1}}y) \in A$ and $A$ is monotone, we have that $\langle J_{A^{-1}}x - J_{A^{-1}}y, J_Ax - J_Ay \rangle \geq 0$, so
$$\langle x-y,J_Ax-J_Ay \rangle \geq 0 + \eps^2/4 \geq \alpha_\psi(\eps).$$
If $\|J_{A^{-1}}x-J_{A^{-1}}y\| \geq \eps/2$, then, since $(J_Ax,J_{A^{-1}}x)$, $(J_Ay,J_{A^{-1}}y) \in A$ and $A$ is inverse uniformly monotone with modulus $\psi$, we have that $\langle J_{A^{-1}}x - J_{A^{-1}}y, J_Ax - J_Ay \rangle \geq \psi(\eps/2)$, so
$$\langle x-y,J_Ax-J_Ay \rangle \geq \psi(\eps/2)+0 \geq \alpha_\psi(\eps).$$
The proof is now finished.
\end{proof}

The following proposition is a quantitative version of \cite[Proposition 4.1]{LiuMouVan22}.

\begin{proposition}\label{p3}
Define, for any suitable $\alpha$, $\eps$,
$$\beta_\alpha(\eps) := \begin{cases} \min\left(\alpha(1)/4,\alpha(\eps)\right), &\mbox{if  $\eps<1$,} \\ 
\alpha(1)/4, & \mbox{if $\eps\geq1$}. \end{cases}$$
Let $\alpha:(0,\infty) \to (0,\infty)$ and  $B : X\to X$ be uniformly monotone with modulus $\alpha$. Then, for all $\eps>0$ and $x$, $y \in X$ with $\|x-y\|\geq\eps$, we have that $\langle x-y,Bx-By \rangle \geq \beta_\alpha(\eps)\|x-y\|^2$.
\end{proposition}

\begin{proof}
Let $\eps>0$ and $x$, $y \in X$ with $\|x-y\| \geq \eps$.

If $\|x-y\|<1$, then $\eps<1$ and, using the uniform monotonicity of $B$, we have that
$$\langle x-y,Bx-By \rangle \geq \alpha(\eps) \geq \alpha(\eps)\|x-y\|^2 \geq \beta_\alpha(\eps)\|x-y\|^2.$$

If $\|x-y\|\geq 1$, then it is enough to show that
$$\langle x-y,Bx-By \rangle \geq \frac{\alpha(1)}4 \|x-y\|^2.$$
Assume that
$$\langle x-y,Bx-By \rangle <\frac{\alpha(1)}4 \|x-y\|^2.$$
We have that there is a $k$ such that $2^k \leq \|x-y\| < 2^{k+1}$. Since $\|x-y\| < 2^{k+1}$, we have that $\frac{\alpha(1)}4 \|x-y\|^2 \leq 2^{2k}\alpha(1)$, so $\langle x-y,Bx-By \rangle <2^{2k}\alpha(1)$. Let $k \in \N$ be minimal such that there are $x$, $y \in X$ such that $2^k \leq \|x-y\|$ and $\langle x-y,Bx-By \rangle <2^{2k}\alpha(1)$. Assume that $k>0$.

Let $z:=(x+y)/2$, so $x-z=z-y=(x-y)/2$ and $\|x-z\|$, $\|z-y\| \geq 2^{k-1}$ (taking note that $k-1\geq 0$). Therefore,
\begin{align*}
\langle x-z,Bx-Bz \rangle + \langle z-y,Bz-By \rangle &= \left\langle\frac{x-y}2,Bx-Bz\right\rangle + \left\langle \frac{x-y}2,Bz-By \right\rangle\\
& = \left\langle \frac{x-y}2,Bx-By\right \rangle < 2^{2k-1}\alpha(1) = 2^{2(k-1)}\alpha(1) + 2^{2(k-1)}\alpha(1),
\end{align*}
so either $\langle x-z,Bx-Bz \rangle <2^{2(k-1)}\alpha(1)$ or $\langle z-y,Bz-By \rangle< 2^{2(k-1)}\alpha(1)$, contradicting the minimality of $k$.

Thus, $k=0$ and so $\|x-y\| \geq 1$ and $\langle x-y,Bx-By \rangle <\alpha(1)$. But this contradicts the fact that $B$ is uniformly monotone with modulus $\alpha$.
\end{proof}

The following proposition is a quantitative version of \cite[Lemma 4.4.(ii)]{LiuMouVan22}.

\begin{proposition}\label{p4}
Let $\alpha_\bullet$ be defined as in Proposition~\ref{p2} and $\beta_\bullet$ as in Proposition~\ref{p3}. Let $\psi:(0,\infty) \to (0,\infty)$ and $A$ be a maximally monotone operator on $X$ which is inverse uniformly monotone with modulus $\psi$. Let $\eps>0$ and $x$, $y \in X$ with $\|x-y\| \geq \eps$. Then
$$ \langle J_{A^{-1}}x - J_{A^{-1}}y, J_Ax - J_Ay \rangle + \|J_Ax - J_Ay\|^2 \geq \beta_{\alpha_\psi}(\eps)\|x-y\|^2.$$
\end{proposition}

\begin{proof}
By Proposition~\ref{p2}, $J_A : X \to X$ is uniformly monotone with modulus $\alpha_\psi$, so, by Proposition~\ref{p3}, we get that
$$\langle x-y,J_Ax-J_Ay \rangle\beta_{\alpha_\psi}(\eps)\|x-y\|^2.$$
Using \eqref{p1}, we obtain the conclusion.
\end{proof}

The following proposition is a quantitative version of \cite[Theorem 4.5.(i)]{LiuMouVan22}.

\begin{proposition}\label{p5}
Let $\alpha_\bullet$ be defined as in Proposition~\ref{p2} and $\beta_\bullet$ as in Proposition~\ref{p3}. Define, for any suitable $\psi$, $\eps$,
$$L_\psi(\eps) := \max\left(\frac{4\eps}{\sqrt{1+4\beta_{\alpha_\psi}(\eps)} -1 }, 2\eps \right).$$
Let $\psi:(0,\infty) \to (0,\infty)$ and $A$ be a maximally monotone operator on $X$ which is inverse uniformly monotone with modulus $\psi$. Then, for all $\eps>0$ and $x$, $y \in X$ with $\|Ax-Ay\| \geq L_\psi(\eps)$, we have that $\|x-y\| > \eps$.
\end{proposition}

\begin{proof}
We shall drop the indices in the sequel. Assume that there are $\eps>0$ and $x$, $y \in X$ with $\|Ax-Ay\| \geq L(\eps)$ and $\|x-y\| \leq\eps$. Now, we have that
$$\|(x+Ax)-(y+Ay)\| \geq \|Ax-Ay\| - \|x-y\| \geq L(\eps)-\eps \geq 2\eps-\eps = \eps.$$
Taking into account that $J_A(x+Ax)=x$, $J_A(y+Ay)=y$, $J_{A^{-1}}(x+Ax)=Ax$ and $J_{A^{-1}}(y+Ay)=Ay$, we have, by Proposition~\ref{p4}, that
$$\langle Ax-Ay,x-y\rangle + \|x-y\|^2 \geq \beta(\eps)\|(x+Ax)-(y+Ay)\|^2,$$
so, using the monotonicity of $A$,
\begin{align*}
\|Ax-Ay\|\|x-y\| + \|x-y\|^2 &\geq \beta(\eps)(\|x-y\|^2 + \|Ax-Ay\|^2 + 2\langle x-y,Ax-Ay \rangle) \\
&\geq \beta(\eps)(\|x-y\|^2 + \|Ax-Ay\|^2 ).
\end{align*}
Set
$$a:=\frac{\|x-y\|}{\|Ax-Ay\|} \leq \frac\eps{L(\eps)} \leq \frac{\sqrt{1+4\beta(\eps)} -1 }4 < \frac{\sqrt{1+4\beta(\eps)} -1 }2,$$
so we have that
$$a\|Ax-Ay\|^2 + a^2\|Ax-Ay\|^2 \geq \beta(\eps)(a^2\|Ax-Ay\|^2 + \|Ax-Ay\|^2),$$
i.e. that $a+a^2 \geq \beta(\eps)(a^2+1)$.
On the other hand, we have that
$$a+ a^2  < \frac{2\sqrt{1+4\beta(\eps)}-2}4 +\frac{1+4\beta(\eps)+1-2\sqrt{1+4\beta(\eps)}}4 = \beta(\eps) \leq \beta(\eps)(a^2+1),$$
a contradiction.
\end{proof}

The following proposition is a quantitative version of \cite[Theorem 4.5.(iv)]{LiuMouVan22}.

\begin{proposition}\label{p6}
Let $L_\bullet$ be defined as in Proposition~\ref{p5}. Define, for any suitable $\psi$, $\eps$,
$$\gamma_\psi(\eps) := \min\left(\frac{\psi(\eps)}{L_\psi(\eps)}, \eps \right).$$
Let $\psi:(0,\infty) \to (0,\infty)$ and $A$ be a maximally monotone operator on $X$ which is inverse uniformly monotone with modulus $\psi$. Then, for all $\eps>0$ and $x$, $y \in X$ with $\|x-y\|<\gamma_\psi(\eps)$, we have that $\|Ax-Ay\| < \eps$ (that is, $\gamma_\psi$ is a modulus of uniform continuity for $A$).
\end{proposition}

\begin{proof}
We have that $\|x-y\| < \gamma_\psi(\eps) \leq \eps$, so, by Proposition~\ref{p5}, $\|Ax-Ay\|<L_\psi(\eps)$. Thus,
$$\langle Ax-Ay,x-y \rangle \leq \|Ax-Ay\|\|x-y\| \leq \|Ax-Ay\| \cdot \gamma_\psi(\eps) < L_\psi(\eps) \cdot \gamma_\psi(\eps) \leq \psi(\eps),$$
so, since $A$ is inverse uniformly monotone with modulus $\psi$, we get that $\|Ax-Ay\| < \eps$.
\end{proof}

\section{Supercoercivity, rectangularity and asymptotic regularity}\label{sec:rect}

The following proposition will refine the analysis of Proposition~\ref{prop-m} into a form which is suitable for the supercoercive case; it was inspired by the usual characterization of uniform smoothness in Banach spaces (treated from the point of view of proof mining in \cite{KohLeu12}).

\begin{proposition}
Let $(M,A,B)$ be adequate in the sense of Proposition~\ref{prop-m}. Let $\eta: (0,\infty) \to (0,\infty)$. TFAE:
\begin{enumerate}[(a)]
\item for all $N>0$ and all $s \geq \eta(N)$, $\vp_{(M,A,B)}(s) \geq Ns$;
\item there exists a classical modulus $\vp$ for $(M,A,B)$ such that for all $N>0$ and all $s \geq \eta(N)$, $\vp(s) \geq Ns$;
\item there exists a modulus $\psi$ for $(M,A,B)$ such that for all $N>0$ and all $s \geq \eta(N)$ such that there is an $x \in M$ with $A(x) \geq s$, $\psi(s) \geq Ns$;
\item for all $N>0$, all $s\geq \eta(N)$ and all $x \in M$ with $A(x) \geq s$, $B(x) \geq Ns$;
\item for all $N>0$ and all $x \in M$ with $A(x) \geq \eta(N)$, $B(x) \geq N \cdot A(x)$.
\end{enumerate}
If this holds, we say that $\eta$ is a {\bf supercoercivity modulus} for $(M,A,B)$.
\end{proposition}

\begin{proof}
For `$(a)\Rightarrow(b)$', just take $\vp:=\vp_{(M,A,B)}$.

For `$(b)\Rightarrow(c)$', define $\psi$ as in the proof of Proposition~\ref{prop-m}.(ii), $(b)\Rightarrow(c)$. Let $N>0$ and $s \geq \eta(N)$ such that there is an $x \in M$ with $A(x) \geq s$. Then $\psi(s)=\vp(s) \geq Ns$.

For `$(c)\Rightarrow(d)$', let $N>0$, $s \geq \eta(N)$ and $x \in M$ with $A(x) \geq s$. Since $\psi$ is a modulus, we have that $\psi(s) \geq B(x)$, so $B(x) \geq Ns$.

For `$(d)\Rightarrow(e)$', just take $s:=A(x)$.

For `$(e)\Rightarrow(a)$', let $N>0$ and $s \geq \eta(N) >0$. We want to show that for every $x \in M$ with $A(x) \geq s$, $B(x) \geq Ns$. Since $A(x) \geq \eta(N)$, $B(x) \geq N \cdot A(x) \geq Ns$.
\end{proof}

We shall say that a uniformly monotone operator $A$ on $X$ is {\it supercoercively uniformly monotone} (traditionally called `uniformly monotone with a supercoercive modulus', see, e.g., \cite[p. 386]{BauCom17}) if $A$ admits a classical modulus of uniform monotonicity $\vp$ such that $\lim_{s\to\infty} \frac{\vp(s)}s = \infty$. By the proposition above, this is equivalent to the fact that there exists a function $\eta : (0,\infty) \to (0,\infty)$, which we shall call a {\it supercoercivity modulus} for $A$, such that for all $N>0$ and all $(a,b)$, $(c,d) \in A$ with $\|a-c\| \geq \eta(N)$, $\langle a-c,b-d \rangle \geq N\|a-c\|$.

Similarly, we shall say that a super strongly nonexpansive mapping $T:X\to X$ is {\it supercoercively super strongly nonexpansive} if $T$ admits an SSNE-modulus $\chi$ such that $\lim_{s\to\infty} \frac{\chi(s)}s = \infty$. Again, by the proposition above (modifying a bit the definition of $\psi$ -- our $\chi$ -- in the proof so that the limit exists regardless of the branch), this is equivalent to the fact that there exists a function $\nu : (0, \infty) \to (0,\infty)$, which we shall call a {\it supercoercivity modulus} for $T$, such that for all $M>0$ and all $x$, $y \in X$ with $\|x-y\|^2 - \|Tx-Ty\|^2 < M \|(x-y)-(Tx-Ty)\|$, we have that $\|(x-y)-(Tx-Ty)\| < \nu(M)$.

By Proposition~\ref{av-ssne}, we know that, for any $\alpha \in (0,1)$, an $\alpha$-averaged mapping has $\chi(\eps):=\eps^2 \cdot \frac{1-\alpha}\alpha$ as an SSNE-modulus and thus we may see that it admits the supercoercivity modulus $\nu(M):=M \cdot \frac{\alpha}{1-\alpha}$, since for all $M>0$ and all $s \geq \nu(M)$, $\chi(s) = s^2  \cdot \frac{1-\alpha}\alpha \geq Ms$.

\begin{proposition}
Let $K:(0,\infty) \to [0,1)$ and $T:X\to X$ such that $T$ is a contraction for large distances with modulus $K$, so, by Proposition~\ref{cld-ssne}, $T$ is super strongly nonexpansive. Then $T$ admits the supercoercivity modulus
$$\nu(M):= \max\left(2, \frac{4M}{1-(K(1))^2}\right).$$
\end{proposition}

\begin{proof}
We will set $\beta:=K(1)$. Let $M>0$ and $x$, $y \in X$ with $\|x-y\|^2 - \|Tx-Ty\|^2 < M \|(x-y)-(Tx-Ty)\|$ and assume that $\|(x-y)-(Tx-Ty)\| \geq \nu(M) >0$. Since
$$\frac12\|(x-y)-(Tx-Ty)\| \leq \frac12(\|x-y\|+\|Tx-Ty\|) \leq \|x-y\|,$$
we have that $\|x-y\| \geq \frac12 \nu(M) \geq 1$, so $\|Tx-Ty\| \leq \beta\|x-y\|$. We get that
\begin{align*}
M \|(x-y)-(Tx-Ty)\| &> \|x-y\|^2 - \|Tx-Ty\|^2 \\
&\geq (1-\beta^2)\|x-y\|^2 \\
&\geq \frac14 (1-\beta^2)\|(x-y)-(Tx-Ty)\|^2,
\end{align*}
so
$$M>  \frac14 (1-\beta^2)\|(x-y)-(Tx-Ty)\| \geq \frac14 (1-\beta^2)\nu(M) \geq \frac14 (1-\beta^2) \cdot  \frac{4M}{1-\beta^2} = M,$$
a contradiction.
\end{proof}

\begin{proposition}
Let $A$ be a maximally monotone operator on $X$ which is inverse uniformly monotone, so, by Proposition~\ref{ium-ssne}, $R_A$ is super strongly nonexpansive. Let $\eta:(0, \infty) \to (0,\infty)$ be a  supercoercivity modulus for $A$. Then $R_A$ admits the supercoercivity modulus $\nu(M):=2\eta(M/2)$.
\end{proposition}

\begin{proof}
Let $M>0$ and $x$, $y \in X$ with $\|x-y\|^2 - \|R_Ax- R_Ay\|^2 < M\|(x-y)-(R_Ax-R_Ay)\|$. Assume that $\|(x-y)-(R_Ax-R_Ay)\| \geq \eta(M)$, so
$$\left\|\frac{x-R_Ax}2 - \frac{y-R_Ay}2 \right\| \geq \frac{\nu(M)}2 = \eta\left(\frac{M}2\right).$$
Using \eqref{ara}, we get that
$$\left\langle \frac{x+R_Ax}2 - \frac{y+R_Ay}2, \frac{x-R_Ax}2 - \frac{y-R_Ay}2 \right\rangle\geq \frac{M}2 \left\|\frac{x-R_Ax}2 - \frac{y-R_Ay}2 \right\|,$$
so
\begin{align*}
\|x-y\|^2 - \|R_Ax- R_Ay\|^2 &= \langle (x-y) + (R_Ax-R_Ay),(x-y) - (R_Ax-R_Ay)\rangle \\
&\geq 4 \cdot \frac{M}2 \left\|\frac{x-R_Ax}2 - \frac{y-R_Ay}2 \right\|\\
& = M \|(x-y)-(R_Ax-R_Ay)\|,
\end{align*}
a contradiction.
\end{proof}

\begin{proposition}\label{etanu}
For every $\nu:(0,\infty) \to (0,\infty)$ and every $N>0$, define
$$\eta_\nu(N):=\nu(2N)/2.$$
Let $A$ be a maximally monotone operator on $X$ which is inverse uniformly monotone, so, by Proposition~\ref{ium-ssne}, $R_A$ is super strongly nonexpansive. Let $\nu:(0, \infty) \to (0,\infty)$ be a  supercoercivity modulus for $R_A$. Then $A$ admits the supercoercivity modulus $\eta_\nu$.
\end{proposition}

\begin{proof}
Let $N>0$ and $(a,b)$, $(c,d) \in A$ with $\|b-d\| \geq \eta(N)$. Assume that $\langle a-c,b-d \rangle < N\|b-d\|$. Using \eqref{ara}, we get, as in the proof of Proposition~\ref{ssne-ium}, that there are $x$, $y \in X$ such that
$$a=\frac{x+R_Ax}2,\ b=\frac{x-R_Ax}2,\ c=\frac{y+R_Ay}2,\ d=\frac{y-R_Ay}2,$$
so
$$\left\langle \frac{x+R_Ax}2 - \frac{y+R_Ay}2, \frac{x-R_Ax}2 - \frac{y-R_Ay}2 \right\rangle< N \left\|\frac{x-R_Ax}2 - \frac{y-R_Ay}2 \right\|$$
and
\begin{align*}
\|x-y\|^2 - \|R_Ax- R_Ay\|^2 &= \langle (x-y) + (R_Ax-R_Ay),(x-y) - (R_Ax-R_Ay) \rangle \\
&< 4N \left\|\frac{x-R_Ax}2 - \frac{y-R_Ay}2 \right\| \\
&= 2N \|(x-y)-(R_Ax-R_Ay)\|.
\end{align*}
We thus get that $ \|(x-y)-(R_Ax-R_Ay)\| < \nu(2N) = 2\eta(N)$, but, since $(x-y)-(R_Ax-R_Ay) = 2(b-d)$, we have that $\|b-d\|<\eta(N)$, a contradiction.
\end{proof}

As in the previous section, we notice that the associations of moduli in the above two propositions are mutually inverse.

A set-valued operator $A \subseteq X \times X$ is called {\it rectangular} -- or {\it $3^*$-monotone} -- if for any $c$ in the domain of $A$ and any $b'$ in the range of $A$, $\sup_{(a,a')\in A} \langle a-c,b'-a'\rangle <\infty$. Quantitative versions of rectangularity have implicitly appeared in \cite{Koh19b} and \cite{Sip22}, while, recently, the notion of `modulus of uniform rectangularity' has been explicitly formalized in \cite{KohPis}. The following proposition generalizes \cite[Proposition 2.1]{Sip22}, and expresses quantitatively the fact (for which see \cite[Exemple 2]{BreHar76} and \cite[Example 25.15]{BauCom17}) that supercoercively uniformly monotone operators are rectangular.

\begin{proposition}\label{prop-theta}
Put, for all $\eta:(0,\infty) \to (0,\infty)$, $L_1$, $L_2$, $L_3 > 0$,
\begin{align*}
\rho(\eta,L_1,L_2,L_3) &:= 2L_3 + L_1L_3 + \eta(2L_1+2L_2+2)\\
\Theta(\eta,L_1,L_2,L_3)&:=(L_1+L_2) (L_3 + \rho(\eta,L_1,L_2,L_3)).
\end{align*}
Let $\eta:(0,\infty) \to (0,\infty)$, $L_1$, $L_2$, $L_3 > 0$ and $A: X \to X$ be inverse uniformly monotone with supercoercivity modulus $\eta$. Let $b$, $c \in X$ with $\|b\|\leq L_1$, $\|c\|\leq L_2$ and $\|Ab\| \leq L_3$. Then, for all $a \in X$,
$$\langle a-c,Ab-Aa \rangle \leq \Theta(\eta,L_1,L_2,L_3).$$
\end{proposition}

\begin{proof}
Let $a \in X$. We shall write $\rho$ for $\rho(\eta,L_1,L_2,L_3)$, so $\Theta(\eta,L_1,L_2,L_3)=(L_1+L_2)(L_3 +\rho)$. If $\|Aa\| \leq \rho$, then, using the fact that $A$ is monotone, so $\langle b-a, Ab-Aa \rangle  \geq 0$, we have that
$$\langle a-c,Ab-Aa \rangle \leq \langle b-c, Ab-Aa \rangle \leq (\|b\|+\|c\|)(\|Ab\|+\|Aa\|) \leq (L_1 + L_2)(L_3 + \rho).$$
In the case where $\|Aa\| \geq \rho >0$, we have that
$$\|Aa-Ab\| \geq \|Aa\| - \|Ab\| \geq \rho - L_3 \geq \eta(2L_1+2L_2+2),$$
so
$$\langle Aa-Ab, a-b \rangle \geq (2L_1+2L_2+2)\|Aa-Ab\|.$$
We know that $\|Aa\| \geq \rho \geq L_1L_3 \geq \|Ab\|\|b\|$ and that $\|Aa\| \geq \rho \geq 2L_3 \geq 2\|Ab\|$, so
$$\|Aa-Ab\| \geq \|Aa\| - \|Ab\| = \|Aa\| \left(1- \frac{\|Ab\|}{\|Aa\|}\right) \geq \|Aa\|\left(1-\frac12\right) = \frac{\|Aa\|}2$$
and
\begin{align*}
\langle Aa-Ab,a \rangle &\geq (2L_1+2L_2+2)\|Aa-Ab\| + \langle Aa-Ab,b \rangle \\
&\geq (2L_1+2L_2+2)\|Aa-Ab\| - \|Aa\| \|b\| - \|Ab\| \|b\| \\
&\geq (2L_1+2L_2+2)\|Aa-Ab\| - L_1 \|Aa\| - \|Ab\| \|b\| \\
&\geq (L_1+L_2+1)\|Aa\| - L_1\|Aa\| - \|Aa\| = L_2\|Aa\| \geq \|c\|\|Aa\|.
\end{align*}

On the other hand,
$$\langle Aa-Ab,c \rangle \leq \|c\|(\|Aa\|+\|Ab\|),$$
so
$$\langle Aa-Ab, a-c \rangle \geq - \|c\|\|Ab\|,$$
i.e.
$$\langle a-c, Ab-Aa \rangle \leq \|c\|\|Ab\| \leq L_2L_3 \leq (L_1+L_2)(L_3+\rho).$$
\end{proof}

The following theorem is a generalization of \cite[Theorem 2.2]{Sip22}. In the proof, we crucially exploit the fact that the original argument was asymmetrical in $R_1$ and $R_2$.

\begin{theorem}\label{thm-phi}
Let $\psi_\bullet$ be defined as in Proposition~\ref{ssne-ium}, $L_\bullet$ as in Proposition~\ref{p5}, $\eta_\bullet$ as in Proposition~\ref{etanu} and $\Theta$ as in Proposition~\ref{prop-theta}. Put, for any $\chi$, $\nu:(0,\infty) \to (0,\infty)$, $\delta > 0$ and $K : (0,\infty) \to (0, \infty)$,
\begin{align*}
B(\nu,K,\delta)&:=\sqrt{\left(K\left(\frac\delta4\right)+\frac\delta8\right)^2 + 2\Theta\left(\eta_\nu,K\left(\frac\delta4\right)+\frac\delta8,K\left(\frac\delta4\right)+\frac\delta8,\frac\delta8\right)}\\
G(\nu,K,\delta)&:=B(\nu,K,\delta) \cdot \max\left(\sqrt{2}, \frac{4B(\nu,K,\delta)}\delta \right)\\
H(\chi,\nu,K,\delta)&:=L_{\psi_\chi}\left(G(\nu,K,\delta) + K\left(\frac\delta4\right)+\frac\delta8 \right)\\
\Phi(\chi,\nu,K,\delta)&:=G(\nu,K,\delta) + H(\chi,\nu,K,\delta) + \frac\delta8.
\end{align*}
Let $\chi$, $\nu:(0,\infty) \to (0,\infty)$ and $R_1$, $R_2:X \to X$ be super strongly nonexpansive mappings such that $R_1$ has SSNE-modulus $\chi$ and $R_2$ has supercoercivity modulus $\nu$. Put $R:=R_2 \circ R_1$. Let $K: (0,\infty) \to (0,\infty)$ be such that for all $i$ and all $\eps > 0$ there is a $p \in X$ with $\|p\| \leq K(\eps)$ and $\|p-R_ip\| \leq \eps$.

Then for all $\delta > 0$ there is a $p \in X$ with $\|p\| \leq \Phi(\chi,\nu,K,\delta)$ and $\|p-Rp\| \leq \delta$.
\end{theorem}

\begin{proof}
Let $\delta > 0$. We know that there are two single-valued maximally monotone and inverse uniformly monotone operators, $A$, $B : X \to X$, such that $R_1 = R_A$ and $R_2 = R_B$. We know that $A$ has modulus of inverse uniform monotonicity $\psi_\chi$ and $B$ has supercoercivity modulus $\eta_\nu$.

Put $\eps:=\delta/4$. By the hypothesis, there are $p_1$, $p_2\in X$ such that $\|p_1\|$, $\|p_2\| \leq K(\eps)$ and $\|p_1-R_Ap_1\|$, $\|p_2-R_Ap_2\| \leq \eps$. Since, by the definition of the reflected resolvent, $p_1-R_Ap_1 = 2(p_1-J_Ap_1)$ and $p_2-R_Bp_2=2(p_2-J_Bp_2)$, we have that $\|p_1-J_Ap_1\|$, $\|p_2-J_Bp_2\| \leq \eps/2$. Also, we have, by the definition of the resolvent, that $p_1-J_Ap_1=AJ_Ap_1$ and $p_2-J_Bp_2=BJ_Bp_2$.

Put $f:=p_1-J_Ap_1+p_2-J_Bp_2$,
$$c:=\Theta\left(\eta_\nu, K(\eps)+\frac\eps2, K(\eps)+\frac\eps2, \frac\eps2\right),$$
and
$$\eta:=\min\left(\frac12,\frac{\eps^2}{\left(K(\eps)+\frac\eps2\right)^2 + 2c} \right),$$
so $\eta \in (0,1)$ and
$$\sqrt{\eta}\cdot \sqrt{\left(K(\eps)+\frac\eps2\right)^2 + 2c} \leq\eps.$$

By the sum rule, $A+B$ is maximally monotone. Then, by Minty's theorem, there is an $u \in X$ such that $f=\eta u + Au + Bu$. Since $A$ is monotone,
$\langle Au-(p_1-J_Ap_1), u-J_Ap_1 \rangle \geq 0$. Since $B$ has supercoercivity modulus $\eta_\nu$, and we know that
$$\|J_Bp_2\| \leq \|p_2\|+\|J_Bp_2-p_2\| \leq K(\eps) + \frac\eps2,$$
that similarly,
$$\|J_Ap_1\| \leq K(\eps) + \frac\eps2,$$
and that
$$\|BJ_Bp_2\| = \|p_2-J_Bp_2\| \leq \frac\eps2,$$
we may apply Proposition~\ref{prop-theta} to get that
$$\langle u - J_Ap_1, BJ_Bp_2 - Bu \rangle \leq c,$$
so
$$\langle Bu - (p_2-J_Bp_2), u - J_Ap_1 \rangle \geq -c.$$
Summing up, we get that
$$\langle f-\eta u - f, u - J_Ap_1 \rangle \geq -c,$$
so
$$\langle u,u-J_Ap_1 \rangle \leq c/\eta.$$
On the other hand, we have that
$$\|J_Ap_1\|^2 = \|u-(u-J_Ap_1)\|^2 = \|u\|^2 - 2\langle u,u-J_Ap_1 \rangle + \|u-J_Ap_1\|^2 \geq \|u\|^2 - 2c/\eta + 0,$$
so
$$\|u\|^2 \leq \|J_Ap_1\|^2 + 2c/\eta$$
and
$$\eta \|u\|^2 \leq \eta \|J_Ap_1\|^2 + 2c \leq \|J_Ap_1\|^2 + 2c.$$
Therefore
$$\|\eta u\| = \sqrt{\eta} \cdot \sqrt{\eta} \cdot \|u\| \leq \sqrt{\eta} \cdot \sqrt{\|J_Ap_1\|^2 + 2c} \leq  \sqrt{\eta} \cdot \sqrt{\left(K(\eps) + \frac\eps2\right)^2 + 2c}  \leq \eps,$$
$$\|f-\eta u\| \leq \|f\| + \|\eta u\| \leq \frac\eps2 + \frac\eps2 + \eps = 2\eps.$$

We have that
$$2J_A - \id_X = R_A = \id_XR_A = (2J_B - R_B)R_A = 2J_BR_A - R_BR_A,$$
so
$$2J_A - 2J_BR_A = \id_X -R_BR_A.$$
Now set
$$z:=(2J_A - 2J_BR_A)(u+Au) = (\id_X -R_BR_A)(u+Au).$$
We have (using for the first equality the definition of the resolvent, and for the second the {\it inverse resolvent identity}, \cite[p. 399, (23.17)]{BauCom17}) that
$$Bu = J_{B^{-1}}(u+Bu) = u+Bu - J_B(u+Bu)$$
and, since by the definition of the resolvent, $J_A(u+Au)=u=J_B(u+Bu)$ and, by the definition of the reflected resolvent, $R_A(u+Au)=u-Au$,
$$z=2u-2J_B(u-Au)=2J_B(u+Bu) - 2J_B(u-Au).$$

We may now bound:
$$\|z\| = 2\|J_B(u+Bu) - J_B(u-Au)\| \leq 2\|u+Bu-u+Au\| = 2\|Au+Bu\| = 2\|f-\eta u\| \leq 4\eps = \delta.$$
We may then set $p:=u+Au$, since, as $z=p-Rp$, we have that $\|p-Rp\| \leq \delta$. We now only have to bound $p$.

Since, per the above,
$$\sqrt{\eta} \cdot \|u\| \leq \sqrt{\|J_Ap_1\|^2 + 2c} \leq \sqrt{\left(K(\eps) + \frac\eps2\right)^2 + 2c} = B(\nu,K,\delta),$$
we have, by the definition of $\eta$, that
$$\|u\| \leq B(\nu,K,\delta) \cdot \max\left(\sqrt{2}, \frac{B(\nu,K,\delta)}\eps \right) = G(\nu,K,\delta).$$

Since
$$\|u-J_Ap_1\| \leq \|u\| + \|J_Ap_1\| \leq G(\nu,K,\delta) + K(\eps) + \frac\eps2,$$
and we know that $A$ has $\psi_\chi$ as a modulus of uniform monotonicity, we may apply Proposition~\ref{p5} to get that
$$\|Au-AJ_Ap_1\| \leq L_{\psi_\chi}\left(G(\nu,K,\delta) + K(\eps) + \frac\eps2\right) = H(\chi,\nu,K,\delta).$$
Now we get that
\begin{align*}
\|p\| &\leq \|u\| + \|Au\| \\
&\leq \|u\| + \|Au-AJ_Ap_1\| + \|AJ_Ap_1\| \\
&= \|u\| + \|Au-AJ_Ap_1\| + \|p_1-J_Ap_1\| \\
&\leq G(\nu,K,\delta) + H(\chi,\nu,K,\delta) + \eps/2\\
&= \Phi(\chi,\nu,K,\delta),
\end{align*}
which is what we wanted to show.
\end{proof}

The following theorem is a generalization of \cite[Theorem 2.3]{Sip22}.

\begin{theorem}\label{thm-psi}
Let $\chi_\bullet$ be defined as in Proposition~\ref{prop-comp} and $\Phi$ as in Theorem~\ref{thm-phi}. Define, for all $\delta > 0$, $K : (0,\infty) \to (0, \infty)$ and suitable finite sequences $\{\chi_i\}_i$, $\{\nu_i\}_i \subseteq (0,\infty)^{(0,\infty)}$, $\Psi(2,\{\chi_i\}_{i=1}^1,\{\nu_i\}_{i=2}^2,K,\delta):=\Phi(\chi_1,\nu_2,K,\delta)$ and, for all $m \geq 2$, $\Psi(m+1,\{\chi_i\}_{i=1}^{m},\{\nu_i\}_{i=2}^{m+1},K,\delta)$ to be
$$\Phi(\chi_{\chi_1,\ldots,\chi_m},\nu_{m+1},\rho\mapsto\max(\Psi(m,\{\chi_i\}_{i=1}^{m-1},\{\nu_i\}_{i=2}^m,K,\rho),K(\rho)) ,\delta).$$
Let $m \geq 2$, $\{\chi_i\}_{i=1}^{m-1}$, $\{\nu_i\}_{i=2}^m \subseteq (0,\infty)^{(0,\infty)}$ and $R_1,\ldots,R_m:X \to X$ be super strongly nonexpansive mappings, such that, for all $i \in \{1,\ldots,m-1\}$, $R_i$ has SSNE-modulus $\chi_i$, and, for all $i \in \{2,\ldots,m\}$, $R_i$ has supercoercivity modulus $\nu_i$. Put $R:=R_m \circ\ldots\circ R_1$. Let $K: (0,\infty) \to (0,\infty)$ be such that for all $i$ and all $\eps > 0$ there is a $p \in X$ with $\|p\| \leq K(\eps)$ and $\|p-R_ip\| \leq \eps$.

Then for all $\delta > 0$ there is a $p \in X$ with $\|p\| \leq \Psi(m,\{\chi_i\}_{i=1}^{m-1},\{\nu_i\}_{i=2}^m,K,\delta)$ and $\|p-Rp\| \leq \delta$.
\end{theorem}

\begin{proof}
As in the proof of \cite[Theorem 2.3]{Sip22}, this follows by simple induction on $m$, using Theorem~\ref{thm-phi} for both the base step and the induction step and the fact that, by Proposition~\ref{prop-comp}, for each $l$, $R_l \circ \ldots\circ  R_1$ has SSNE-modulus $\chi_{\chi_1,\ldots,\chi_l}$.
\end{proof}

As in \cite{Sip22}, we shall need the following result, which quantitatively connects, for strongly nonexpansive mappings, the approximate fixed point property to asymptotic regularity.

\begin{theorem}[{cf. \cite[Theorem 1]{Koh19b}, \cite[Theorem 2.5]{Sip22}}]\label{vp}
Define, for any $\eps$, $b$, $d > 0$, $\alpha : (0, \infty) \to (0, \infty)$ and $\omega:(0,\infty) \times (0,\infty) \to (0,\infty)$,
$$\Gamma(\eps,b,d,\alpha,\omega):=\left\lceil \frac{18b + 12\alpha(\eps/6)}\eps-1\right\rceil \cdot \left\lceil \frac{d}{\omega\left(d,\frac{\eps^2}{27b + 18\alpha(\eps/6)}\right)} \right\rceil.$$
Let $T:X\to X$ and $\omega:(0,\infty) \times (0,\infty) \to (0,\infty)$ such that $T$ is strongly nonexpansive with modulus $\omega$. Let $\alpha : (0, \infty) \to (0, \infty)$ such that for any $\delta > 0$ there is a $p \in X$ with $\|p\|\leq\alpha(\delta)$ and $\|p-Tp\| \leq \delta$. Then for any $\eps$, $b$, $d>0$ and any $x \in X$ with $\|x\| \leq b$ and $\|x-Tx\|\leq d$, we have that for any $n \geq \Gamma(\eps,b,d,\alpha,\omega)$, $\|T^nx-T^{n+1}x\|\leq \eps$.
\end{theorem}

Putting together the above results, we obtain the following generalization of \cite[Theorem 2.8]{Sip22}.

\begin{theorem}\label{thm-main}
Let $\omega_\bullet$ be defined as in Proposition~\ref{omega-bullet}, $\chi_\bullet$ as in Proposition~\ref{prop-comp}, $\Psi$ as in Theorem~\ref{thm-psi} and $\Gamma$ as in Theorem~\ref{vp}. Define, for all $m \geq 2$, $\eps$, $b$, $d > 0$, $K : (0,\infty) \to (0, \infty)$, $\{\chi_i\}_{i=1}^m$, $\{\nu_i\}_{i=2}^m \subseteq (0,\infty)^{(0,\infty)}$,
$$\Sigma_{m,\{\chi_i\}_{i=1}^m,\{\nu_i\}_{i=2}^m,K,b,d}(\eps):=\Gamma\left(\eps,b,d,\delta\mapsto \Psi(m,\{\chi_i\}_{i=1}^{m-1},\{\nu_i\}_{i=2}^m,K,\delta),\omega_{\chi_{\chi_1,\ldots,\chi_m}}\right).$$
Let $m \geq 2$, $\{\chi_i\}_{i=1}^m$, $\{\nu_i\}_{i=2}^m \subseteq (0,\infty)^{(0,\infty)}$ and $R_1,\ldots,R_m:X \to X$ be super strongly nonexpansive mappings, such that, for all $i$, $R_i$ has SSNE-modulus $\chi_i$, and, for all $i \in \{2,\ldots,m\}$, $R_i$ has supercoercivity modulus $\nu_i$. Put $R:=R_m \circ\ldots\circ R_1$. Let $K: (0,\infty) \to (0,\infty)$ be such that for all $i$ and all $\eps > 0$ there is a $p \in X$ with $\|p\| \leq K(\eps)$ and $\|p-R_ip\| \leq \eps$.

Then, for any $b>0$, $d>0$ and any $x \in X$ with $\|x\| \leq b$ and $\|x-Rx\|\leq d$, we have that $\Sigma_{m,\{\chi_i\}_{i=1}^m,\{\nu_i\}_{i=2}^m,K,b,d}$ is a rate of asymptotic regularity for the sequence $(R^nx)$ w.r.t. $R$, i.e., for any $\eps>0$ and $n \geq \Sigma_{m,\{\chi_i\}_{i=1}^m,\{\nu_i\}_{i=2}^m,K,b,d}(\eps)$,
$$\|R^nx-R^{n+1}x\|\leq \eps.$$
\end{theorem}

Since the corresponding qualitative result is also new here, we present it below in full.

\begin{theorem}\label{thm-qual}
Let $X$ be a Hilbert space, $m \geq 1$ and $R_1,\ldots,R_m:X \to X$ be supercoercively super strongly nonexpansive mappings which have the approximate fixed point property.

Then $R_m \circ\ldots\circ R_1$ is asymptotically regular.
\end{theorem}

\section{Acknowledgements}

I would like to thank Ulrich Kohlenbach for suggesting the final form of Proposition~\ref{prop-comp}; and Lauren\c tiu Leu\c stean for his comments which led to an improvement of the presentation.

\end{document}